\documentclass[11pt]{amsart}
\usepackage[a4paper,left=3cm,right=2cm,top=2.5cm,bottom=2.5cm]{geometry}

\usepackage[T1, T5]{fontenc}
\usepackage{charter}
\usepackage{xcolor}
\usepackage{latexsym,amssymb,amsmath,graphicx,hyperref,setspace,tikz,subcaption,float,enumerate,comment}

\usetikzlibrary{patterns}

\usetikzlibrary{positioning}
\tikzset{main node/.style={circle,fill=blue!15,draw,minimum size=.5cm,inner sep=0pt},
}

\newtheorem{theorem}{Theorem}[section]
\newtheorem{lemma}[theorem]{Lemma}
\newtheorem{proposition}[theorem]{Proposition}
\newtheorem{corollary}[theorem]{Corollary}

\theoremstyle{definition}
\newtheorem{definition}[theorem]{Definition} 
\newtheorem{remark}[theorem]{Remark}
\newtheorem{example}[theorem]{Example}

\newcommand{\lr}[1]{\left\langle#1\right\rangle}
\newcommand{\mm}[0]{\mathfrak{m}}
\newcommand{\K}[0]{\mathbb{K}}
\newcommand{\f}{\mathcal{F}}

\DeclareMathOperator{\dep}{depth}
\DeclareMathOperator{\reg}{reg}
\DeclareMathOperator{\pdim}{pdim}
\DeclareMathOperator{\bight}{bight}
\DeclareMathOperator{\link}{link}
\DeclareMathOperator{\del}{del}
\DeclareMathOperator{\dis}{dis}

\title{Sequentially Cohen--Macaulay Co-Chordal Graphs: Structure and Projective Dimension}
\date{%January 20, 2022
}

\author{Chwas Ahmed, Amir Mafi and Mohammed Rafiq Namiq*}
\address{Chwas Ahmed, Department of Mathematics, College of Science, University of Sulaimani, Kurdistan Region, Iraq.}
\email{chwas.ahmed@univsul.edu.iq}

\address{A. Mafi, Department of Mathematics, University of Kurdistan, P.O. Box: 416, Sanandaj, Iran.}
\email{a\_mafi@ipm.ir}

\address{Mohammed Rafiq Namiq, Department of Mathematics, College of Science, University of Sulaimani, Kurdistan Region, Iraq.}
\email{mohammed.namiq@univsul.edu.iq}

\makeatletter
\@namedef{subjclassname@2020}{%
	\textup{2020} Mathematics Subject Classification}
\makeatother

\keywords{Sequentially Cohen--Macaulay graphs, Co-chordal graphs, Chordal graphs, Projective dimension, Maximum degree, $d$-tree, Simplicial complex, Edge ideal, Graph classification.\\
* Corresponding author}
\begin{document}

\begingroup
\def\uppercasenonmath#1{} % this disables uppercasing title
\let\MakeUppercase\relax % this disables uppercasing authors
\maketitle
\endgroup

%%%%%%%%%%%%%%%%%%%%%%%%%%%%%%%%%%%%%%%%%%%%%%%%%%%%%%%%%%%%%%%%%%
% New Commands

%%%%%%%%%%%%%%%%%%%%%%%%%%%%%%%%%%%%%%%%%%%%%%%%%%%%%%%%%%%%%%%% 
 
 \sloppy
\begin{abstract}
	We introduce a class of chordal graphs called $(d_1,d_2,\dots,d_q)$-trees. A graph belongs to this class if and only if its clique complex is sequentially Cohen--Macaulay, providing a complete classification of all sequentially Cohen--Macaulay co-chordal graphs. This class also yields a classification of bi-sequentially Cohen--Macaulay graphs. We study the relationship between the projective dimension of a graph and its maximum vertex degree. We show that the projective dimension is always at least the maximum vertex degree, although this bound is not always tight, even for co-chordal graphs. However, equality holds when the graph is sequentially Cohen--Macaulay co-chordal or has a full vertex.

%In this paper, we introduce a class of graphs called $(d_1,d_2,\dots,d_q)$-tree, which are chordal. We show that the complement of $(d_1,d_2,\dots,d_q)$-tree graphs characterize sequentially Cohen--Macaulayness of co-chordal graphs, which provides a new perspective on the study of sequentially Cohen--Macaulay co-chordal graphs. This class of graphs classifies all bi-sequentially Cohen--Macaulay graphs. We also study the relation between the projective dimension of a graph $G$ and the maximum degree of its vertices. We show that the projective dimension of $G$ is greater than or equal to the maximum degree of its vertices, and that this bound is not necessarily tight even for co-chordal graphs. However, we show that the projective dimension of a graph $G$ is equal to the maximum degree of its vertices when $G$ is sequentially Cohen--Macaulay co-chordal or has a full-vertex.

\end{abstract}

%%%%%%%%%%%%%%%%%%%%%%%%%%%%%%%%%%%%%%%%%%%%%%%%%%%%%%%%%%%%%%%%%%%

\section{Introduction}
Let $G$ be a finite simple graph with vertex set $V(G)=\{x_1,\dots,x_n\}$ and edge set $E(G)=\bigl\{\{x_i,x_j\}\subseteq V(G)\mid x_i\mbox{ is adjacent to }x_j\bigr\}$. We associate to $G$ its edge ideal
$$I(G)=\bigl(x_ix_j\mid\{x_i,x_j\}\in E(G)\bigr)\subseteq R=\K[x_1,\dots,x_n],$$
where $\K$ is a field \cite{Villarreal1990}. A graph $G$ is called Cohen--Macaulay (resp. sequentially Cohen--Macaulay) if the quotient ring $R/I(G)$ has the corresponding property.

A central class of graphs in combinatorial commutative algebra is the class of chordal graphs, which are characterized by the absence of induced cycles of length four or more. Chordal graphs have many remarkable properties, including the fact that all chordal graphs are sequentially Cohen--Macaulay \cite{FranciscoTuyl2007}. In contrast, not every co-chordal graph, a graph whose complement is chordal, is sequentially Cohen--Macaulay, motivating the natural question: which co-chordal graphs possess this property?

Fröberg \cite{Froberg1990} partially addressed this question by introducing $d$-trees, a subclass of chordal graphs, and showing that the complement of a $d$-tree graph is Cohen--Macaulay. To provide a complete answer, we introduce a new class of graphs called $(d_1,d_2,\dots,d_q)$-trees, defined in terms of a non-increasing sequence of positive integers $(d_1,d_2,\dots,d_q)$. We show that every $d$-tree graph is $(d_1,d_2,\dots,d_q)$-tree graph, and every $(d_1,d_2,\dots,d_q)$-tree graph is chordal, that is,
$$d\mbox{-tree graphs }\Rightarrow(d_1,d_2,\dots,d_q)\mbox{-tree graphs }\Rightarrow\mbox{ chordal graphs}.$$
while the converse implications may fail in general. In Theorem \ref{d-tree-VD}, we show that a co-chordal graph $G$ is sequentially Cohen--Macaulay if and only if its complement is a $(d_1,d_2,\dots,d_q)$-tree graph for some sequence $(d_1,d_2,\dots,d_q)$. This provides a practical and effective criterion for identifying sequentially Cohen--Macaulay co-chordal graphs.

Next, we study the relationship between the projective dimension of a graph $G$ and its maximum vertex degree, defined as the largest degree of any vertex in $G$. We show that the projective dimension of $G$ is always bounded above by its maximum degree (Theorem \ref{pdim-max}). While the difference between projective dimension and maximum degree can be arbitrarily large in general, this bound is sharp for certain families of graphs, including those with a full vertex or whose complement is a $(d_1,d_2,\dots,d_q)$-tree (Theorems \ref{pdim-G1-empty} and \ref{d-tree-pdim}).

In the direction of this work, Gitler and Valencia conjectured in {\cite[Conjecture 4.13]{GitlerValencia2005}} that for any connected graph $G$ whose complement is chordal, the projective dimension of $G$ is equal to its maximum degree. Our results show that the conjecture does not hold in general, even though the difference between the projective dimension and maximum degree may not be bounded, (see Section \ref{section}). There have been some attempts to prove the conjecture for some special cases. Gitler and Valencia {\cite[Theorem 4.14]{GitlerValencia2005}} showed that the conjecture holds for some graphs in a class of graphs in which the graph and its complement are both chordal. These graphs are also  $(d_1,d_2,\dots,d_q)$-trees. Similarly, Moradi and Kiani {\cite[Theorem 1.1]{MoradiKiani2010}} proved the conjecture when the complement of $G$ is a $d$-tree. Therefore, our Theorem \ref{pdim-G1-empty} extends these results by showing that the conjecture holds for a broader class of graphs, namely those whose complement is a $(d_1,d_2,\dots,d_q)$-tree.

%%%%%%%%%%%%%%%%%%%%%%%%%%%%%%%%%%%%%%%%%%%%%%%%%%%%%%%%%%%%%%%%%%

\section{Preliminaries}
A \emph{simplicial complex} $\Delta$ on the vertex set $V=\{x_1,\ldots,x_n\}$ is a collection of subsets of $V$ such that
\begin{enumerate}
	\item[i)] $\{x_i\}\in\Delta$ for every $1\leq i\leq n$ and
	\item[ii)] if $F\in\Delta$ and $H\subseteq F$, then $H\in\Delta$. 
\end{enumerate}
An element $F$ of $\Delta$ is  a \emph{face} of $\Delta$ and a maximal (with respect to inclusion) face is a \emph{facet}. The set of all facets of $\Delta$ is denoted by $\f(\Delta)$ and we sometimes write $\lr{F\mid F\in\f(\Delta)}$ for $\Delta$. A simplicial complex $\Delta$ is called \emph{pure} if the facets have the same cardinality. Let $d=\max\{|F|\mid F\in\Delta\}$, the \emph{dimension} of $\Delta$ is $\dim\Delta=d-1$. The \emph{Stanley--Reisner ideal} of $\Delta$ is $$I_\Delta=\lr{x_F\mid F\notin\Delta}$$ 
where $x_F=\prod_{x_i\in F}x_i$. The quotient algebra $\K[\Delta]=R/I_{\Delta}$ is the \emph{Stanley--Reisner ring} of $\Delta$ over a field $\K$. The Krull dimension of $\K[\Delta]$ is $\dim\K[\Delta]=\dim\Delta+1$.
A simplicial complex with only one facet is a \emph{simplex}.
Let $f_i=f_i(\Delta)$ be the number of faces of $\Delta$ of cardinality $i+1$. The sequence $f(\Delta)=(f_{-1}=1,f_0,f_1,\dots,f_{d-1})$ is the \emph{$f$-vector} of $\Delta$. The \emph{$h$-vector} $h(\Delta)=(h_0,h_1,\dots,h_d)$ of $\Delta$ can be computed in terms of $f$-vector as follows:
\begin{equation}\label{eq h_i}
	h_i=\sum_{j=0}^{i}(-1)^{i-j}\binom{d-j}{i-j}f_{j-1},\quad 0\leq i\leq d.
\end{equation}
%$$\sum_{i=0}^d f_{i-1}(t-1)^{d-i}=\sum_{i=0}^d h_i t^{d-i}.$$
The \emph{Hilbert series} of $\K[\Delta]$ is of the form $H_{\K[\Delta]}(t)=(h_0+h_1t+\dots+h_st^s)/(1-t)^d$. The \emph{Hilbert polynomial}  of $\K[\Delta]$ is $P_{\K[\Delta]}(t)=h_0+h_1t+\dots+h_st^s$ with $h_s\neq0$. The \emph{$a$-invariant} $a(\K[\Delta])$ is the degree of rational function $H_{\K[\Delta]}(t)$, that is, $s-d$.

A \emph{subcomplex} $\Gamma$ of $\Delta$ is a simplicial complex whose facets are faces of $\Delta$. If $F \in \Delta$ is a face, then the \emph{deletion} of $F$ is the subcomplex
$$\del_{\Delta}(F)=\{E\in\Delta\mid E\cap F= \emptyset\},$$
of $\Delta$ and the \emph{link} of $F$ is the subcomplex of $\Delta$ and defined by
$$\link_{\Delta}(F)=\{E\in\Delta\mid E\cap F=\emptyset\text{ and }E\cup F\in\Delta\}.$$

Let $G$ be a graph and $I(G)$ the edge ideal of $G$ as defined in the introduction section.  A \emph{path} $P_k$ of length $k$ in $G$ is a sequence of distinct vertices $x_{i_0}, x_{i_1}, \dots, x_{i_k}$ and a sequence of edges $\{x_{i_j}, x_{i_{j+1}}\}\in E(G)$. A \emph{cycle}  of length $k$ in $G$ is a path $P_k$ together with the edge $\{x_{i_k}, x_{i_0}\}$. 
%
% A \emph{path} $P_k$ of length $k$ is a graph with vertices $V(P_k)=\{v_1,v_2, \dots,v_k\} $ and edges $E(P_k)=\{\{v_1,v_2\}, \{v_2,v_3\},\dots, \{v_{k-1},v_k\}\}$
%
% sequence of edges $e_1,e_2,\dots,e_k$ of $G$ such that $|V(e_i)\cap V(e_{i+1})|=1$  
%A \emph{path} $P_k$ is a graph on $k$ vertices whose edge set is $E(P_k)=\{\{v_{i},v_{i+1}\}\mid1\leq i\leq k-1\text{ and } v_j\in V(G)\}$. 
A \emph{connected} graph $G$ is a graph that has a path between every pair of vertices, and disconnected otherwise.
 The \emph{neighbourhood} of a vertex $x\in V(G)$ is the set $N_G(x)=\{x_j\in V(G)\mid x\mbox{ is adjacent to }x_j\}$. 
 For any subset $F\subseteq V(G)$, the set of neighbourhood of $F$ is $N_G(F)=\bigcup_{x\in F}N_G(x)$. The \emph{closed neighbourhood} of $F$ is $N_G[F]=N_G(F)\cup F$. The \emph{degree} of a vertex $x\in V(G)$ is defined to be $\deg_G(x)=|N_G(x)|$. An \emph{isolated} vertex of $G$ is a vertex of degree zero.
 A graph $H$ is a \emph{subgraph} of $G$ if $V(H)\subseteq V(G)$ and $E(H)\subseteq E(G)$. An \emph{induced subgraph} on $S\subseteq V(G)$ is a subgraph $G_S$ of $G$ such that $\{x_i,x_j\}\in E({G_S})$ if $\{x_i,x_j\}\in E(G)$ for all $x_i,x_j\in S$. The \emph{deletion} $\del_G(S)$ of $S$ is an indued subgraph of $G$ on $V(G)\backslash S$. The \emph{disjointness} $\dis_G(S)$ of $S$ is an induced subgraph of $G$ on $V(G)\backslash N_G[S]$. If $S=\{x\}$, we write $\del_G(x)$ and $\dis_G(x)$ instead of $\del_G(\{x\})$ and $\dis_G(\{x\})$, respectively. Hence $\dis_G(S)=\del_G(N_G[S])$. 
 
 The \emph{complement} of a graph $G$ is the graph $\overline{G}$ such that $V(G)=V({\overline{G}})$ and $E({\overline{G}})=\{\{x_i,x_j\}\mid \{x_i,x_j\}\notin E(G)\}$. A subset $C \subseteq V(G)$ is called a \emph{vertex cover} of $G$ if every edge of $G$ has one its endpoints in $C$.  A vertex cover is \emph{minimal} if it is minimal with respect to set inclusion among the set of vertex covers $G$. A graph $G$ is \emph{unmixed} if all minimal vertex covers have the same cardinality. A subset $F\subseteq V(G)$ is called an \emph{independent} set of $G$ if no two vertices in $F$ are adjacent. An independent set $F$ is \emph{maximal} if it is not contained in any other independent set. A graph $G$ is well-covered if all maximal independent sets have the same cardinality. Hence, a graph $G$ is unmixed if and only if it is well-covered. A subset $K\subseteq V(G)$ is a \emph{clique} of $G$ if every two distinct vertices of $K$ are adjacent in $G$. A \emph{complete graph} $\mathcal{K}_r$ is a clique on $r$ vertices.The \emph{independence complex} $\Delta_G$ of $G$ is the set of all independent sets of $G$, that is $$\Delta_G=\{F\subseteq V(G)\mid F\text{ is an independent set in }G\}.$$
In this case $I_{\Delta_G}=I(G)$. The \emph{clique complex} $\Delta(G)$ of $G$ consists all the clique sets of $G$, that is $$\Delta(G)=\{K\subseteq V(G)\mid K\text{ is a clique set of }G\}.$$
Note that the independent simplicial complex of a graph $G$ is the clique complex of $\overline{G}$. Hence $\Delta_G=\Delta(\overline{G})$. Further, a graph $G$ is unmixed if and only if $\Delta_G$ is pure. An \emph{induced $k$-cycle} $C_k$ in $G$ is cycle of length $k$ such that $G_{V(C_k)}=C_k$. A graph $G$ is called \emph{chordal} if it does not contain any $C_k$, $k\geq4$. A graph $G$ is  \emph{co-chordal} if $\overline{G}$ is chordal.

Consider the following minimal graded free resolution of $\K[\Delta_G]=R/I(G)$ over $R$
$$
0 \longrightarrow \bigoplus_{j} R(-j)^{\beta_{p, j}} \longrightarrow \cdots \longrightarrow \bigoplus_{j} R(-j)^{\beta_{2, j}}\longrightarrow \bigoplus_{j} R(-j)^{\beta_{1, j}} \longrightarrow R \longrightarrow\K[\Delta_G] \longrightarrow 0
$$
where $R(-j)$ denotes the $R$-module obtained by shifting the degrees of $R$ by $j$.
The integer $\beta_{i, j}(\K[\Delta_G]):=\beta_{i, j}$ is called the $i^{th}$ graded Betti number of $\K[\Delta_G]$ in degree $j$. The length $p$ of the resolution is called the \emph{projective dimension} of $\K[\Delta_G]$ over $R$, that is
$$
\pdim\K[\Delta_G]=\max \left\{i\mid \beta_{i, j}(\K[\Delta_G]) \neq 0 \text { for some } j\right\} \text {. }
$$
The (Castelnuovo–Mumford) \emph{regularity} of $\K[\Delta_G]$ over $R$ is $$\reg\K[\Delta_G]=\max\{j-i\mid \beta_{i,j}(\K[\Delta_G])\neq0\}.$$

The 2-linear resolution (over $R$) of $\K[\Delta_G]$ is the minimal graded free resolution of the form
$$
0 \longrightarrow R(-2-p)^{\beta_{p}} \longrightarrow \cdots\longrightarrow R(-3)^{\beta_{1}} \longrightarrow R(-2)^{\beta_{0}} \longrightarrow R \longrightarrow\K[\Delta_G] \longrightarrow 0.
$$

Faridi \cite{Faridi2004} proposed the concept of a leaf for a simplicial complex $\Delta$ and simplicial forests, which was inspired by the definition of trees and leaves in graph theory. A tree is a connected graph with no cycles. Alternatively, a connected graph is a tree if every subgraph has a vertex that is connected to only one edge of the graph, which is called a leaf. A facet $F$ of a simplicial complex $\Delta$ is a \emph{leaf} if either $F$ is the only facet in $\Delta$, or there is another facet $M$ in $\Delta$ (different from $F$) such that for every facet $N\in\Delta$ (excluding $F$), $N\cap F\subset M\cap F$. After Faridi, Zheng \cite{Zheng2004} introduced the notion of quasi-forest simplicial complexes. A simplicial complex $\Delta$ is quasi-forest if there is an order $F_1,\dots,F_q$ of the facets of $\Delta$, called a \emph{leaf order}, such that $F_i$ is a leaf of the subcomplex $\lr{F_1,\dots,F_i}$ for each $i=1,\dots,q$. A connected quasi-forest simplicial complex is \emph{quasi-tree}. A \emph{free vertex} is a vertex which belongs to exactly one facet. 

The concept of sequentially Cohen--Macaulayness was first introduced by Stanley {\cite[Definition 2.9]{Stanley1996}}. Stanley showed that sequentially Cohen--Macaulayness is a weaker property than Cohen--Macaulayness, but still has many important applications in algebraic combinatorics and algebraic geometry. Bj\"orner and Wachs extended the definitions of vertex decomposability and shellability for non-pure cases, see {\cite[Definition 2.1]{BjornerWachs1996}} and {\cite[Definition 11.1]{BjornerWachs1997}}. A simplicial complex $\Delta$ is  \emph{vertex decomposable} if $\Delta$ is an empty set, $\Delta$ is a simplex or there exists a vertex $v \in V$ such that
\begin{enumerate}
	\item[i)]  $\del_{\Delta}(v)$ and $\link_{\Delta}(v)$ are both vertex decomposable.
	\item[ii)]\label{shedding} No facet $F$ of $\link_{\Delta}(v)$ is also a facet of $\del_{\Delta}(v)$ (equivalently, every facet of $\del_\Delta(v)$ is a facet of $\Delta$). 
\end{enumerate}
If $\Delta$ is pure, we call $\Delta$ {pure vertex decomposable}. A vertex $v$ that satisfies condition (ii) is called \emph{shedding vertex}. A simplicial complex $\Delta$ is \emph{shellable}  if the facets of $\Delta$ can be ordered, say $F_1,\ldots,F_r$, such that for all $1\leq i<j\leq r$, there  exists some $v\in F_j\setminus F_i$ and some $1\leq k<j$ with $F_j\setminus F_k= \{v\}$. Such ordering  $F_1,\ldots,F_r$  is called a \emph{shelling order}. If $\Delta$ is pure, we call $\Delta$ {pure shellable}. A finitely generated graded $R$-module $N$ is \emph{sequentially Cohen--Macaulay} if there exists a finite filtration of $R$-modules
$$0=N_0\subset N_1\subset\dots\subset N_r=N$$
such that each quotient $N_i/N_{i-1}$ is Cohen--Macaulay and $\dim N_1/N_0<\dim N_2/N_1<\dots<\dim N_r/N_{r-1}$. Bj\"orner and Wachs showed that if $\Delta$ is a vertex decomposable, then $\Delta$ is shellable {\cite[Theorem 11.3]{BjornerWachs1997}}. Stanley showed that if $\Delta$ is shellable, then $\Delta$ is sequentially Cohen--Macaulay {\cite[p.87]{Stanley1996}}. Hence, if $\Delta$ is vertex decomposable, then $\Delta$ is sequentially Cohen--Macaulay. The converse is also true by \cite[Proposition 1.2 and Theorem 3.1]{GoodarziYassemi2012} if $\Delta$ is quasi-forest simplicial complex.
Dirac's theorem \cite{Dirac1961} on chordal graphs says that $\Delta_G$ is quasi-forest if and only if $G$ is co-chordal. Hence by {\cite[Theorem 1]{Froberg1990}}, $\Delta_G$ is quasi-forest if and only if $\K[\Delta_G]$ has 2-linear resolution.

We refer the reader to \cite{HerzogHibi2011, Stanley1996, Villarreal2015} for further details regarding the terminologies in this section.

\section{The $(d_1,d_2,\dots,d_q)$-tree graphs}
The classes of $d$-tree and generalized $d$-tree graphs were initially defined by Fr\"oberg in {\cite{Froberg1990}}. In fact, generalized $d$-tree graphs are exactly the chordal graphs. Here, we define a new class of graphs that lie strictly between $d$-trees and chordal graphs. We call these graphs \emph{$(d_1,d_2,\dots,d_q)$-trees}.

\begin{definition}\label{d-tree definition}
	Let $(d_1,d_2,\dots,d_q)$ be a non-increasing sequence of positive integers. Then a \emph{$(d_1,d_2,\dots,d_q)$-tree} is a graph $G$ constructed inductively  as follows:
	\begin{enumerate}
		\item[i)] The graph $H_1=\mathcal{K}_{d_1}$,
		\item[ii)] $H_i=H_{i-1}\bigcup_{\mathcal{K}_{d_i-1}}\mathcal{K}_{d_i}$ for $2\leq i\leq q$, and
		\item[iii)] $G=H_q$.
	\end{enumerate}
	\begin{comment}
		\begin{enumerate}
		\item[i)] The graph $H_1$ is a complete graph on $d_1$ vertices,
		\item[ii)] For $2\le i\le q$, the graph $H_i$ is constructed from $H_{i-1}$ by adding a complete graph on $d_i$ vertices to $H_{i-1}$ from a complete graph of $d_i-1$ vertices, and
		\item[iii)] $G=H_q$.
	\end{enumerate}
	\end{comment}
\end{definition}

From the construction of $(d_1,d_2,\dots,d_q)$-trees, we observe that a $d$-tree is a $(d_1,d_2,\dots,d_q)$-tree with $d_ i=d$ for $1\le i\le q$. Also, a $(d_1,d_2,\dots,d_q)$-tree is a chordal graph. In other words, we have the following implications:
\begin{center}
	$d\mbox{-tree graphs}\Rightarrow(d_1,d_2,\dots,d_q)\mbox{-tree graphs}\Rightarrow\mbox{chordal graphs}$.
\end{center}
However, the converse of the implications do not hold in general. For example, see the graphs in Figure \ref{fig2}.
\begin{figure}[ht]
	\begin{subfigure}{.45\textwidth}
		\begin{center}
			\begin{tikzpicture}[x = 2cm, y = 1.5cm]
				\draw[fill] (0,0) circle [radius = 0.04];
				\draw[fill] (1,1) circle [radius = 0.04];
				\draw[fill] (1,0) circle [radius = 0.04];
				\draw[fill] (2,1) circle [radius = 0.04];
				
				\draw (0,0)--(1,1)--(1,0)--(0,0);
				\draw (1,0)--(1,1)--(2,1)--(1,0);
			\end{tikzpicture}
			\caption*{A $(3,3)$-tree}
		\end{center}
	\end{subfigure}
	\begin{subfigure}{.5\textwidth}
		\begin{center}
			\begin{tikzpicture}[x = 2cm, y = 1.5cm]
				\draw[fill] (0,0) circle [radius = 0.04];
				\draw[fill] (.5,1) circle [radius = 0.04];
				\draw[fill] (1,0) circle [radius = 0.04];
				\draw[fill] (2,0) circle [radius = 0.04];
				
				\draw (0,0)--(.5,1)--(1,0)--(0,0);
				\draw (1,0)--(2,0);
			\end{tikzpicture}
			\caption*{A $(3,2)$-tree but not $d$-tree.}
		\end{center}
	\end{subfigure}
	
	\begin{subfigure}{1\textwidth}
		\begin{center}
			\begin{tikzpicture}[x = 2cm, y = 1.5cm]
				
				\draw[fill] (0,0) circle [radius = 0.04];
				\draw[fill] (.5,1) circle [radius = 0.04];
				\draw[fill] (1,0) circle [radius = 0.04];
				\draw[fill] (2,0) circle [radius = 0.04];
				\draw[fill] (1.5,1) circle [radius = 0.04];
				
				\draw (0,0)--(.5,1)--(1,0)--(0,0);
				\draw (1,0)--(1.5,1)--(2,0)--(1,0);
			\end{tikzpicture}
			\caption*{A chordal graph that is  not a $(d_1,d_2,\dots,d_q)$-tree.}
		\end{center}
	\end{subfigure}
	\caption{Examples of generalised $d$-tree graphs}\label{fig2}
\end{figure}

Note that we can always find a $(d_1,d_2,\dots,d_q)$-tree graph for any sequence $(d_1,d_2,\dots,d_q)$ in the definition \ref{d-tree definition}. However, there may be multiple non-isomorphic $(d_1,d_2,\dots,d_q)$-trees for a given sequence. Figure \ref{fig5} shows examples of two non-isomorphic $(d_1,d_2,\dots,d_q)$)-tree graphs.
\begin{figure}[ht]
	\begin{subfigure}{.55\textwidth}
		\begin{center}
			\begin{tikzpicture}[x = 2cm, y = 1.5cm]
				\draw[fill] (0,0) circle [radius = 0.04];
				\draw[fill] (1,1) circle [radius = 0.04];
				\draw[fill] (1,0) circle [radius = 0.04];
				\draw[fill] (2,1) circle [radius = 0.04];
				\draw[fill] (3,1) circle [radius = 0.04];
				
				\draw (0,0)--(1,1)--(1,0)--(0,0);
				\draw (1,0)--(1,1)--(2,1)--(1,0);
				\draw (2,1)--(3,1);
			\end{tikzpicture}
		\end{center}
	\end{subfigure}
	\begin{subfigure}{.35\textwidth}
		\begin{center}
			\begin{tikzpicture}[x = 2cm, y = 1.5cm]
				\draw[fill] (0,0) circle [radius = 0.04];
				\draw[fill] (1,1) circle [radius = 0.04];
				\draw[fill] (1,0) circle [radius = 0.04];
				\draw[fill] (2,1) circle [radius = 0.04];
				\draw[fill] (2,0) circle [radius = 0.04];
				
				\draw (0,0)--(1,1)--(1,0)--(0,0);
				\draw (1,0)--(1,1)--(2,1)--(1,0);
				\draw (1,0)--(2,0);
			\end{tikzpicture}
		\end{center}
	\end{subfigure}
	\caption{Two non-isomorphic  $(3,3,2)$-tree graphs.}\label{fig5}
\end{figure}

The clique complex of a $(d_1,d_2,\dots,d_q)$-tree graph is $\Delta(G)=\lr{F_1,\dots,F_q}$ where $F_i=V(\mathcal{K}_{d_i})$ for $1\leq i\leq q$. In the next theorem, we show that the clique complex of $(d_1,d_2,\dots,d_q)$-tree graphs are sequentially Cohen--Macaulay.

\begin{theorem}\label{d-tree-VD}
	Let $G$ be a graph. Then $\overline{G}$ is a $(d_1,d_2,\dots,d_q)$-tree if and only if $\Delta_G=\lr{F_1,F_2,\dots,F_q}$ is a vertex decomposable (hence shellable and sequentially Cohen--Macaulay) quasi-forest simplicial complex.	
\end{theorem}
\begin{proof}
	Let $F_i = V(\mathcal{K}_{d_i})$ for $i=1,\dots,q$. Using the definition of $(d_1,d_2,\dots,d_q)$-tree, we see that the complex $\Delta_G = \langle F_1, F_2,\dots,F_q\rangle$ is the clique complex of $\overline{G}$. By Dirac's theorem, see \cite[Theorem 1 and 2]{Dirac1961}, the graph $\overline{G}$ is chordal. Therefore, we can conclude that $\Delta_G$ is a quasi-forest.
	
	For the rest of the proof we use induction on $|V(G)|$ to show that $\Delta_G$ is vertex decomposable. If $|V(G)|=2$, then $G$ is an edge. Clearly, $\overline{G}$ is a $(1,1)$-tree and $\Delta_G$ is vertex decomposable. Suppose $|V(G)|=n>2$ for some integer $n$, and the statement is true for any graph with less than $n$ vertices.
	
	By the definition of $(d_1,d_2,\dots,d_q)$-tree, there is a vertex $x\in F_q\backslash\bigcup_{i=1}^{q-1}F_i$ and $F_q\backslash F_j=\{x\}$ for some $1\leq j<q-1$. Hence, $x$ is a shedding vertex.
	
	Then, $\del_{\Delta_G}(x)=\lr{F_1,\dots,F_{q-1}}$ and $\link_{\Delta_G}(x)=\lr{F_q\backslash\{x\}}$. Note that $\del_{\overline{G}}(x)$
	is a $(d_1,\dots,d_{q-1})$-tree and $\dis_{\overline{G}}(x)$ is a $(d_{q}-1)$-tree. On the other hand, $\del_{\Delta_G}(x)=\Delta_{\del_G(x)}$ and $\link_{\Delta_G}(x)=\Delta_{\dis_{G}(x)}$. Hence, by the induction hypothesis, $\del_{\Delta_G}(x)$ and $\link_{\Delta_G}(x)$ are vertex decomposable. Therefore, $\Delta_G$ is vertex decomposable.
	
	Conversely, we assume that $\Delta_{G}$ is a shellable quasi-forest simplicial complex with shelling order ${F_1,{F}_{2},\dots,F_{q}}$. We use induction on the number of facets of $\Delta_{G}$.
	%Assume that the statement is true for every quasi-forest with less than $q$ facets and $\Delta_{G}$ is a shellable quasi-forest with shelling order $F_1,F_2,\dots,F_q$. 
	By \cite[Lemma 1.1]{GoodarziYassemi2012},, there exists $j>1$  such that $F_j$ is a leaf of $\Delta_{G}$ with a unique free vertex, say $x$.  Let $\Delta_{G'}=\lr{F_1,\dots,\widehat{F}_{j},\dots,F_{q}}, 1<j\leq q$. Then $\Delta_{G'}$ is shellable quasi-forest with shelling order $F_1,\dots,\widehat{F}_{j},\dots,F_{q}$ (for more details, see the proof of Proposition 1.2 in \cite{GoodarziYassemi2012}). The inductive step implies that $\overline{G'}$ is a $(d_1,\dots,\widehat{d}_j,\dots,d_q)$-tree for $1< j\leq q$. Now, we add the facet $F_j$ to $\Delta_{G'}$.
	By the shellablity of $\Delta_{G}$, there exists a facet $F_i$ with $1\leq i<j$ such that $F_j\backslash F_i=\{x\}$. This implies that $\overline{G}$ is a $(d_1,d_2,\dots,d_q)$-tree. By \cite[Proposition 1.2]{GoodarziYassemi2012}, shellablity and vertex decomposability are equivalent.
\end{proof}

\begin{remark}\label{remark 2}
	\begin{enumerate}[\rmfamily i)]
		\item The sequence of positive integers $(d_1,d_2,\dots,d_q)$ can be used to identify sequentially Cohen--Macaulay quasi-forest Stanley--Reisner rings from Theorem \ref{d-tree-VD}.
		\item Let  $\Delta_G=\lr{F_1,\dots, F_q}$ be the independence complex of a graph $G$ such that $\overline{G}$ is a $(d_1,\dots, d_q)$-tree. Then from the definition of shelling and leaf order we note that $F_1,\dots, F_q$ is a shelling and a leaf order of $\Delta_G$.
	\end{enumerate}
\end{remark}

The following Corollary follows immediately form Theorem \ref{d-tree-VD}.

\begin{corollary}\label{d-tree pure VD}
	The clique complex of a $d$-tree is pure vertex decomposable quasi-forest and vice versa (hence pure shellable {\cite[Theorem 2.13]{MoradiKiani2010}} and Cohen--Macaulay {\cite[Theorem 2]{Froberg1990}}).
\end{corollary}

In the following we provide some necessary conditions for a graph in order to be a $(d_1,d_2,\dots,d_q)$-tree. 

\begin{proposition}\label{qf b shedding}
	Let $G$ be a graph such that $\overline{G}$ is a $(d_1,d_2,\dots,d_q)$-tree. Then any vertex of maximum degree in $G$ is a free and shedding vertex of $\Delta_G$.
\end{proposition}
\begin{proof}
	Let $F_{i}=V\left(\mathcal{K}_{d_{i}}\right)$ for $i=1,\dots,q$. By the definition of $(d_1, d_2, \ldots, d_q)$-tree, we can construct the independence complex $\Delta_G=\lr{F_1,F_2,\dots,F_q}$ from $G$. Let $x$ be a vertex in $G$ with the maximum degree. Since $\Delta_G$ is the independence complex of $G$, $x$ has the minimum degree in $\overline{G}$. Since $F_q$ contains a free vertex, and it is of minimum cardinality in $\Delta_G$, we have $\left|F_j\right| = \left|F_q\right|$ whenever $x \in F_j$, for $1 \leq j \leq q$. Hence, if $x$ is not a free vertex, then the degree of $x$ cannot be minimum in $\overline{G}$, which is a contradiction. Therefore, $x$ must be a free vertex.
	
	To show $x$ is a shedding vertex. Suppose first that $x\in F_1$. Then we have $|F_1|=|F_j|$ for all $j$ since $F_1$ contains a free vertex and has minimum cardinality. Now suppose that $x$ is not a shedding vertex. Then $F_1\backslash\{x\}$ is not a face of  $\left\langle  F_{2}, \ldots, F_{q}\right\rangle$. This implies that $\mathcal{K}_{d_{2}}$ is attached to $\mathcal{K}_{d_{1}}$ from at most $\mathcal{K}_{d_{2}-2}$, which contradicts the fact that $\overline{G}$ is a $(d_1,d_2,\dots,d_q)$-tree. Now, if $x \in F_j$ for some $2 \leq j \leq q$, then there exists $1\leq i<j$ such that $F_j \setminus F_i=\{x\}$. In either case, we have shown that $x$ is a shedding vertex.
\end{proof}

\begin{proposition}\label{qf max adjacent}
	If $G$ is a graph such that $\overline{G}$ is a $(d_1,d_2,\dots,d_q)$-tree, then any two vertices of $G$ with maximum degree are adjacent.
\end{proposition}
\begin{proof}
	Suppose that $G$ has two non-adjacent vertices $x$ and $x'$ of maximum degree. We want to show that $N_G(x)=N_G(x')$, which means that $x$ and $x'$ have the same set of neighbours. Assume that there is a vertex $v'$ in $N_G(x')\backslash N_G(x)$. Then $d(x,x')\geq2$ and $d(x,v')\geq2$. This means that $\dis_G(x)$ contains at least the edge $\{x',v'\}$. Therefore, $x$ is not a free vertex in $\Delta_G$, which contradicts Proposition \ref{qf b shedding}.
	
	Thus, we have $N_G(x)=N_G(x')$. Any facet $F$ of $\Delta_G$ that contains $x$ must also contain $x'$. Therefore, $F\backslash\{x\}$ is a facet of $\del_{\Delta_G}(x)$. This implies that $x$ is not a shedding vertex, which again contradicts Proposition \ref{qf b shedding}. Therefore, it follows that the vertices in $G$ with the maximum degrees must be adjacent.
\end{proof}

\begin{example}
	Consider the following applications of  Propositions~\ref{qf b shedding} and \ref{qf max adjacent}. 
	\begin{itemize}
		\item Consider the graph $G$ and its complement in Figure \ref{fig6}. The vertex $x$ has maximum degree in $G$, but it is not a free vertex in the independence complex $\Delta_G$. Therefore, according to Proposition \ref{qf b shedding}, $\overline{G}$ cannot be a $(d_1,d_2,\dots,d_q)$-tree. Furthermore, $G$ has more that two non-adjacent vertices of maximum degrees, Proposition \ref{qf max adjacent} also implies that $\overline{G}$ is not a $(d_1,d_2,\dots,d_q)$-tree.
		\begin{figure}[ht]
			\begin{subfigure}{.45\textwidth}
				\begin{center}
					\begin{tikzpicture}[x = 2cm, y = 1.5cm]
						\draw[fill] (0,0) circle [radius = 0.04];
						\draw[fill] (.7,0) circle [radius = 0.04];
						\draw[fill] (.7,.7) circle [radius = 0.04];
						\draw[fill] (0,.7) circle [radius = 0.04];
						\draw[fill] (1.25,.35) circle [radius = 0.04];
						\draw[fill] (-.55,.35) circle [radius = 0.04];
						\draw[fill] (.35,1.4) circle [radius = 0.04];
						
						\draw (0,0)--(.7,0)--(.7,.7)--(0,.7)--(0,0);
						\draw (.7,0)--(1.25,.35);
						\draw (0,0)--(-.55,.35)--(0,.7)--(0,0);
						\draw (.35,1.4)--(.7,.7);
						\draw (.35,1.4)--(1.25,.35);
						\draw (.35,1.4)--(-.55,.35);
						\draw (0,.7)--(1.25,.35);
						\draw (0,0)--(1.25,.35);
						\draw (-.55,.35)--(.7,0);
						
						\node[below] at (0,0) {$x$};
					\end{tikzpicture}
					\caption*{A graph $G$}
				\end{center}
			\end{subfigure}
			\begin{subfigure}{.45\textwidth}
				\begin{center}
					\begin{tikzpicture}[x = 2cm, y = 1.5cm]
						\draw[fill] (0,0) circle [radius = 0.04];
						\draw[fill] (.5,1) circle [radius = 0.04];
						\draw[fill] (1,0) circle [radius = 0.04];
						\draw[fill] (2,0) circle [radius = 0.04];
						\draw[fill] (3,0) circle [radius = 0.04];
						\draw[fill] (3.5,1) circle [radius = 0.04];
						\draw[fill] (4,0) circle [radius = 0.04];
						
						\draw (0,0)--(.5,1)--(1,0)--(0,0);
						\draw (1,0)--(2,0);
						\draw (2,0)--(3,0);
						\draw (3,0)--(3.5,1)--(4,0)--(3,0);
						
						\node[above] at (2,0) {$x$};
					\end{tikzpicture}
				\end{center}
				\caption*{The complement of the graph $G$}
			\end{subfigure}
			\caption{A non-$(d_1,d_2,\dots,d_q)$-tree graph}\label{fig6}
		\end{figure}
		
		\item The complement of induced 4-cycle $C_4$ is not a $(d_1,d_2,\dots,d_q)$-tree because it has two non-adjacent vertices of maximum degrees.
		
		\item Consider the graph $G$ in Figure \ref{fig12}. Then the complement of the graph $G$ is not $(d_1,d_2,\dots,d_q)$-tree because it contains two vertices, namely $x_2$ and $x_4$, that have maximum degrees but are not adjacent in $G$.
		\begin{figure}[ht]
			\begin{center}
				\begin{tikzpicture}[x = 2cm, y = 1.5cm]
					\draw[fill] (0,0) circle [radius = 0.04];
					\draw[fill] (.75,.75) circle [radius = 0.04];
					\draw[fill] (.75,0) circle [radius = 0.04];
					\draw[fill] (.75,-.75) circle [radius = 0.04];
					\draw[fill] (1.5,0) circle [radius = 0.04];
					\draw[fill] (2.25,0) circle [radius = 0.04];
					\draw[fill] (1.5,-.75) circle [radius = 0.04];
					
					\draw (0,0)--(.75,.75)--(.75,0)--(.75,-.75)--(0,0);
					\draw (.75,.75)--(1.5,0)--(.75,-.75);
					\draw (.75,0)--(1.5,0)--(2.25,0);
					\draw (.75,.75)--(2.25,0);
					\draw (.75,-.75)--(2.25,0);
					\draw (.75,0)--(1.5,-.75);
					\draw (2.25,0)--(1.5,-.75);
					
					\node[left] at (0,0) {$x_1$};
					\node[above] at (.75,.75) {$x_2$};
					\node[left] at (.75,0) {$x_3$};
					\node[below] at (.75,-.75) {$x_4$};
					\node[above right=-.08] at (1.5,0) {$x_5$};
					\node[right] at (2.25,0) {$x_6$};
					\node[below] at (1.5,-.75) {$x_7$};
				\end{tikzpicture}
				\caption{The graph $G$}\label{fig12}
			\end{center}
		\end{figure}
	\end{itemize}
\end{example}

%%%%%%%%%%%%%%%%%%%%%%%%%%%%%%%%%%%%%%%%%%%%%%%%%%%%%%%%%%%%%%%

\section{Relation between Projective dimension of a graph and its maximum degrees}\label{section}
In this section, we establish a lower bound for $\pdim\K[\Delta_G]$ based on the maximum degree of vertices in any graph $G$. Moreover, we provide a sufficient condition under which $\pdim\K[\Delta_G]$ is equal to $\max_{1\leq i \leq n}\left\{\deg_{G}(x_i)\right\}$. We start by the following standard Lemma:

\begin{lemma}
	Let $G$ be a graph and $F\subseteq V(G)$. Then $F$ is a maximal independent set of $G$ if and only if $N_G[F]=V(G)$ and $F\cap N_G(F)=\emptyset$.
	\label{max ind}
\end{lemma}
\begin{proof}
	($\Rightarrow$)	Let $F$ be a maximal independent set. Then $V(G)\backslash F$ is the set of all neighbourhoods of $F$. Hence $F\cup N_G(F)=V(G)$ and $F\cap N_G(F)=\emptyset$.
	
	($\Leftarrow$) Let $x\in F$. Then $x\notin N_G(F)$ since $F\cap N_G(F)=\emptyset$. Thus there is no edge between the vertices of $F$. Therefore $F$ is an independent set. Now suppose there is an independent set $F'$ of $G$ with $F\subsetneq F'$. Let $y\in F'\backslash F$. Then $y\in N_G(F)$ and so $\{y,x_i\}$ is an edge for some $x_i\in F$. Hence $F\cup\{y\}\subseteq F'$ is not an independent set in $G$ which is contradiction. Therefore $F$ is maximal. 
\end{proof}

\begin{definition}\label{graph G_i}
	Let $F=\{x_1,\dots,x_r\}$ be an ordered independent set of a graph $G$. For $1\leq i\leq r$, the graph $G_i$ is obtained from $G$ inductively as follows:
	\begin{enumerate}
		\item[i)] $G_0=G$,
		\item[ii)] $G_{i}=\dis_{G_{i-1}}(x_i)$.
	\end{enumerate}
\end{definition}

\begin{remark}
	We can observe from Definition \ref{graph G_i} that $N_G(F)=\bigcup_{i=1}^{r}N_{G_{i-1}}(x_i)$ and $N_{G_{i-1}}(x_i)\cap N_{G_{j-1}}(x_j)=\emptyset$ for $1\leq i<j\leq r$.
\end{remark}

\begin{example}\label{example 1}
	Let us consider the graph $G$ depicted in Figure~\ref{Gi-process}. We choose the independent set $F=\{x_1,x_5\}$. Starting from $G$, we construct $G_1=\dis_G(x_1)$ by removing the closed neighbourhood of $x_1$ in $G$. Next, we construct $G_2=\dis_{G_{1}}(x_5)$ by removing the closed neighbourhood of $x_5$ in $G_1$. One can choose the other order, $F=\{x_5,x_1\}$, but the resulting graphs $G_i$ may differ. The graphs $G$, $G_1$, and $G_2$ are illustrated in Figure \ref{Gi-process}.
	\begin{figure}[ht]
		\begin{subfigure}{.4\textwidth}
			\begin{center}
				\begin{tikzpicture}[x = 2cm, y = 1.5cm]
					\draw[fill] (0,0) circle [radius = 0.04];
					\draw[fill] (.75,.75) circle [radius = 0.04];
					\draw[fill] (.75,0) circle [radius = 0.04];
					\draw[fill] (.75,-.75) circle [radius = 0.04];
					\draw[fill] (1.5,0) circle [radius = 0.04];
					\draw[fill] (2.25,0) circle [radius = 0.04];
					\draw[fill] (1.5,-.75) circle [radius = 0.04];
					
					\draw (0,0)--(.75,.75)--(.75,0)--(.75,-.75)--(0,0);
					\draw (.75,.75)--(1.5,0)--(.75,-.75);
					\draw (.75,0)--(1.5,0)--(2.25,0);
					\draw (.75,.75)--(2.25,0);
					\draw (.75,-.75)--(2.25,0);
					\draw (.75,0)--(1.5,-.75);
					\draw (2.25,0)--(1.5,-.75);
					
					\node[left] at (0,0) {$x_1$};
					\node[above] at (.75,.75) {$x_2$};
					\node[left] at (.75,0) {$x_3$};
					\node[below] at (.75,-.75) {$x_4$};
					\node[above right=-.08] at (1.5,0) {$x_5$};
					\node[right] at (2.25,0) {$x_6$};
					\node[below] at (1.5,-.75) {$x_7$};
				\end{tikzpicture}
				\caption*{The graph $G$}
			\end{center}
		\end{subfigure}
		\begin{subfigure}{.3\textwidth}
			\begin{center}
				\begin{tikzpicture}[x = 2cm, y = 1.5cm]
					\draw[fill] (.75,0) circle [radius = 0.04];
					\draw[fill] (1.5,0) circle [radius = 0.04];
					\draw[fill] (2.25,0) circle [radius = 0.04];
					\draw[fill] (1.5,-.75) circle [radius = 0.04];
					
					\draw (.75,0)--(1.5,0)--(2.25,0);
					\draw (.75,0)--(1.5,-.75);
					\draw (2.25,0)--(1.5,-.75);
					
					\node[above] at (.75,.75) {$~$};
					\node[above] at (.75,0) {$x_3$};
					\node[above] at (1.5,0) {$x_5$};
					\node[above] at (2.25,0) {$x_6$};
					\node[below] at (1.5,-.75) {$x_7$};
				\end{tikzpicture}
				\caption*{The graph of $G_1$}
			\end{center}
		\end{subfigure}
		\begin{subfigure}{.25\textwidth}
			\begin{center}
				\begin{tikzpicture}[x = 2cm, y = 1.5cm]
					\draw[fill] (1.5,-.75) circle [radius = 0.04];
					
					\node[above] at (.75,.75) {$~$};
					
					\node[below] at (1.5,-.75) {$x_7$};
				\end{tikzpicture}
				\caption*{The graph of $G_2$}
			\end{center}
		\end{subfigure}
		\caption{A visual representation of the graph $G$ and the intermediate graphs $G_1$ and $G_2$ obtained during the process}\label{Gi-process}
	\end{figure}
\end{example}

\begin{lemma}\label{max-process}
	Let $G$ be a graph and $F=\{x_1,\dots,x_r\}$ be an independent set of $G$. Then $G_r=\emptyset$ if and only if $F$ is a maximal independent set of $G$.
\end{lemma}
\begin{proof}
	By Definition \ref{graph G_i}, we have $G_r=\dis_{G_{r-1}}(x_r)=\dis_G(F)$. Therefore, $G_r=\emptyset$ if and only if $\dis_G(F)=\emptyset$ if and only if $N_G[F]=V(G)$ by Lemma \ref{max ind}. Thus, the assertion follows from Lemma \ref{max ind}.
\end{proof}

\begin{lemma}\label{deg neighbourhood}
	Let $G$ be a graph and $F=\{x_1,\dots,x_r\}$ is a maximal independent set of $G$. Then $ |N_{G}(F)|=\sum_{i=1}^{r}\deg_{G_{i-1}}(x_i).$
\end{lemma}
\begin{proof}
	Using the construction of $G_i$ for $1\leq i\leq r$ and applying Lemma \ref{max-process}, we can deduce that $G_r$ is empty. Therefore, by Lemma \ref{max ind} we have $V(G)=N_{G}(x_1)\cup N_{G_1}(x_2)\cup\dots\cup N_{G_{r-1}}(x_r)\cup F$. By comparing this with Lemma \ref{max ind}, we get $N_G(F)=\bigcup_{i=1}^{r}N_{G_{i-1}}(x_i)$. Thus, we can conclude that $|N_G(F)|=\sum_{i=1}^{r}\deg_{G_{i-1}}(x_i)=n-r$.
\end{proof}

\begin{remark}
	We note that $F\in\mathcal{F}(\Delta_G)$ has maximum (minimum) cardinality if and only if $|N_G(F)|$ is minimum (maximum).
\end{remark}

\begin{definition}
	For $1\leq i\leq r$, a \emph{max-process} is a procedure that produces an  ordered independent subset $F$ of $V(G)$   as   follows:
	\begin{enumerate}
		\item[i)] $F_1=\{x_1\in V(G)\mid\deg_G(x_1)\text{ is maximum}\}\text{ and }G_1=\dis_{G}(x_1)$,
		\item[ii)] $F_i=F_{i-1}\bigcup\left\{x_i\in V(G)\mid \deg_{G_{i-1}}(x_i)\text{ is maximum}\right\}\text{ and }G_i=\dis_{G_{i-1}}(x_i)$,
		\item[iii)] $F=F_r$.
	\end{enumerate}
\end{definition}

\begin{remark}
	We can note from Lemma \ref{max-process} that the independent set $F$ produced by the max-process is a maximal independent set of $G$ if and only if $G_r=\emptyset$.
\end{remark}

\begin{example}\label{example 3}
	Let $G$ be the graph shown in Figure \ref{fig8}. We observe that for $i=2,3,4,5,6$, $\deg_G(x_i)$ is maximum in $G$.
	
	If we choose $x_3$, then $F_1=\{x_3\}$ and $G_1=\dis_{G}(x_3)$. As a result, the vertices $x_1$ and $x_6$ have degree zero in $G_1$. We have two choices for the second step $v_2$, let choose $x_1$. Then $F_2=\{x_3,x_1\}$ and $G_2=\dis_{G_{1}}(x_1)$. Finally, $F=\{x_3,x_1,x_6\}$ and $G_3=\emptyset$. Therefore, according to Lemma \ref{max-process}, $F$ is a maximal independent set of $G$.
	\begin{figure}[ht]
		\begin{subfigure}{.35\textwidth}
			\begin{center}
				\begin{tikzpicture}[x = 2cm, y = 1.5cm]
					\draw[fill] (0,0) circle [radius = 0.04];
					\draw[fill] (.75,.75) circle [radius = 0.04];
					\draw[fill] (.75,0) circle [radius = 0.04];
					\draw[fill] (.75,-.75) circle [radius = 0.04];
					\draw[fill] (1.5,0) circle [radius = 0.04];
					\draw[fill] (2.25,0) circle [radius = 0.04];
					\draw[fill] (1.5,-.75) circle [radius = 0.04];
					
					\draw (0,0)--(.75,.75)--(.75,0)--(.75,-.75)--(0,0);
					\draw (.75,.75)--(1.5,0)--(.75,-.75);
					\draw (.75,0)--(1.5,0)--(2.25,0);
					\draw (.75,.75)--(2.25,0);
					\draw (.75,-.75)--(2.25,0);
					\draw (.75,0)--(1.5,-.75);
					\draw (2.25,0)--(1.5,-.75);
					
					\node[left] at (0,0) {$x_1$};
					\node[above] at (.75,.75) {$x_2$};
					\node[left] at (.75,0) {$x_3$};
					\node[below] at (.75,-.75) {$x_4$};
					\node[above right=-.08] at (1.5,0) {$x_5$};
					\node[right] at (2.25,0) {$x_6$};
					\node[below] at (1.5,-.75) {$x_7$};
				\end{tikzpicture}
				\caption*{The graph $G$}
			\end{center}
		\end{subfigure}
		\begin{subfigure}{.37\textwidth}
			\begin{center}
				\begin{tikzpicture}[x = 2cm, y = 1.5cm]
					\draw[fill] (0.75,0) circle [radius = 0.04];
					\draw[fill] (2.25,0) circle [radius = 0.04];
					
					\node[below] at (0.75,0) {$x_1$};
					\node[below] at (2.25,0) {$x_6$};
				\end{tikzpicture}
				\caption*{The graph of $G_1$}
			\end{center}
		\end{subfigure}
		\begin{subfigure}{.25\textwidth}
			\begin{center}
				\begin{tikzpicture}[x = 2cm, y = 1.5cm]
					\draw[fill] (2.25,0) circle [radius = 0.04];
					
					\node[below] at (2.25,0) {$x_6$};
				\end{tikzpicture}
				\caption*{The graph of $G_2$}
			\end{center}
		\end{subfigure}
		\caption{The Graph $G$ and Intermediate Graphs $G_1$ and $G_2$ during the Process}\label{fig8}
	\end{figure}
\end{example}

The \emph{big height} $\bight I(G)$ of an edge ideal $I(G)$ is the maximum cardinality of minimal vertex covers of $G$, equivalently, the cardinality of a set whose complement of a maximal independent set of minimum cardinality.

Moray and Villarreal in Corollary 3.33 \cite{MoreyVillarreal2012} established a lower bound for the projective dimension of $\K[\Delta_G]$, which states that for any graph $G$ we have $\pdim\K[\Delta_G]\geq\bight I(G),$ the equality holds only if $\K[\Delta_G]$ is sequentially Cohen--Macaulay. The next theorem provides a lower bound for the big height of the edge ideal of any graph and this bounded is sharp.

\begin{theorem}\label{pdim-max}
	Let $G$ be a graph. Then $\displaystyle\pdim\K[\Delta_G]\geq\bight I(G)\geq\max_{1\leq i \leq n}\{\deg_G(x_i)\}.$
\end{theorem}
\begin{proof}
	Let $F=\{x_1,\dots,x_r\}$ be a maximal independent set  produced by max-process. Then, the cardinality of $F$ is between the minimum and maximum cardinalities of any maximal independent sets in $G$, denoted by $s$ and $d$, respectively, where $\dim \K[\Delta_G]=d$. Let $x_1$ be the vertex chosen in the first step of the max-process, which has maximum degree in $G$, $\max_{1\leq i \leq n}\{\deg_G(x_i)\}=\deg_{G}(x_1)$. Using Lemma \ref{deg neighbourhood}, we have $\deg_G(x_1)\leq\sum_{i=1}^{r}\deg_{G_{i-1}}(x_i)=|N_G(F)|=n-r$ since $F$ is a maximal independent set of $G$.  By applying {\cite[Corollary 3.33]{MoreyVillarreal2012}}, we obtain $\pdim\K[\Delta_G]\geq\bight I(G)=n-s\geq n-r\geq\max_{1\leq i \leq n}\{\deg_G(x_i)\}$.
\end{proof}

Now, we give an example which shows the difference between the projective dimension and a maximum degree of a graph may not be bounded in general. The following inequality is well known (see {\cite[Corollary B.4.1]{Vasconcelos1998}})
\begin{equation}\label{eq1}
	\deg P_{\K[\Delta_G]}(t)-\reg\K[\Delta_G]\leq\dim\K[\Delta_G]-\dep\K[\Delta_G]
\end{equation}
and the equality holds if $\K[\Delta_G]$ has 2-linear resolution. Hence the equality holds in Equation \eqref{eq1} if $G$ is a co-chordal graph.

The \emph{$r$-barbell} graph is a chordal graph obtained by connecting two copies of a complete graph $\mathcal{K}_r$ by a bridge, for example see Figure \ref{fig1}.
\begin{figure}[ht]
	\begin{center}
		\begin{tikzpicture}[x = 2cm, y = 1.5cm]
			\draw[fill] (0,0) circle [radius = 0.04];
			\draw[fill] (.75,1) circle [radius = 0.04];
			\draw[fill] (1.5,0) circle [radius = 0.04];
			\draw[fill] (.75,-1) circle [radius = 0.04];
			\draw[fill] (2.25,0) circle [radius = 0.04];
			\draw[fill] (3,1) circle [radius = 0.04];
			\draw[fill] (3.75,0) circle [radius = 0.04];
			\draw[fill] (3,-1) circle [radius = 0.04];
			
			\draw (0,0)--(.75,1)--(1.5,0)--(.75,-1)--(0,0);
			\draw (0,0)--(1.5,0);
			\draw (.75,1)--(.75,-1);
			\draw (2.25,0)--(3,1)--(3.75,0)--(3,-1)--(2.25,0);
			\draw (1.5,0)--(2.25,0);
			\draw (2.25,0)--(3.75,0);
			\draw (3,1)--(3,-1);
			
			%			\node[left] at (0,0) {$x_1$};
			%			\node[above] at (.75,1) {$x_2$};
			\node[above right=-.05] at (1.5,0) {$x_1$};
			%			\node[below] at (.75,-1) {$x_4$};
			\node[above left=-.05] at (2.25,0) {$x_2$};
			%			\node[above] at (2.95,1) {$x_6$};
			%			\node[below] at (2.95,-1) {$x_8$};
			%			\node[right] at (3.75,0) {$x_7$};
		\end{tikzpicture}
	\end{center}
	\caption{The $4$-barbell graph}
	\label{fig1}
\end{figure}

For $r\geq3$, let $G$ be the complement of the $r$-barbell graph. Then the independence complex of $G$ is $\Delta_G=\lr{V(\mathcal{K}_r),V(\mathcal{K}_r),\{x_1,x_2\}}$ on $|V(G)|=2r$. The $f$-vector of $\Delta_G$ is $\left(1,2\binom{r}{1},2\binom{r}{2}+1,2\binom{r}{3},2\binom{r}{4},\dots,2\right)$ and $\dim\K[\Delta_G]=r$. From Equation \eqref{eq h_i}, one can find that $h_{r}=0$ and $h_{r-1}=(-1)^{r-1}(-2)$ after some simple calculations. Hence, we have $h_{r-1}\neq0$ and so the degree of the Hilbert polynomial of $\K[\Delta_G]$ is $\deg P_{\K[\Delta_G]}(t)=r-1$. Thus, the equality in Equation \eqref{eq1} gives $\dep\K[\Delta_G]=r-(r-1)+1=2$. The Auslander--Buchsbaum formula, \cite[Theorem 3.1]{AuslanderBuchsbaum1957}, $\pdim\K[\Delta_G]+\dep\K[\Delta_G]=n$ implies $\pdim\K[\Delta_G]=2r-2$. Note that $\max_{1\leq i \leq 2r}\left\{\deg_{G}(x_i)\right\}=r$. Therefore,
$$\displaystyle\pdim\K[\Delta_G]-\displaystyle\max_{1\leq i \leq 2r}\left\{\deg_{G}(x_i)\right\}=r-2.$$

\section{The projective dimension and maximum degree is equal for certain of classes of graphs}

In the rest of this work, we provide two different sufficient conditions on $G$ for which $\pdim\K[\Delta_G]=\max_{1\leq i \leq n}\{\deg_G(x_i)\}.$ The first class of graphs for which the equality holds is the class of graphs with a  full-vertex. A vertex $x\in V(G)$ is called a \emph{full-vertex} if $N_G[x]=V(G)$. A full-vertex is also known  as a universal vertex or a dominating vertex.

\begin{theorem}\label{pdim-G1-empty}
	Let $G$ be a graph with a full-vertex $x$. Then $$\displaystyle\pdim\K[\Delta_G]=\max_{1\leq i \leq n}\{\deg_G(x_i)\}=n-1.$$
\end{theorem}
\begin{proof}
	Lemma \ref{max-process} implies that $F=\{x\}$ is a maximal independent set of minimum cardinality. Using Lemma \ref{deg neighbourhood}, we have $|N_G(x)|=\deg_{G}(x)=\max_{1\leq i \leq n}\{\deg_{G}(x_i)\}=n-1$. Therefore, by Theorem \ref{pdim-max}, we obtain $\pdim\K[\Delta_G]\geq\bight I(G)=n-1$. We now claim that $\pdim\K[\Delta_G]=n-1$. If it were not the case, assume $\pdim\K[\Delta_G]=n$. Then the Auslander--Buchsbaum formula implies that $\dep\K[\Delta_G]=0$. Note that the maximal ideal $\displaystyle\mm$ is an associated prime of $\K[\Delta_G]$  if and only if $\dep\K[\Delta_G]=0$. It follows that $\bight I(G)=n$ which is a contradiction. Hence $\pdim\K[\Delta_G]=\max_{1\leq i \leq n}\{\deg_{G}(x_i)\}=n-1$.
\end{proof}

\begin{example}
	We consider some graphs with a full-vertex. 
	\begin{itemize}
		\item Since every vertex in $\mathcal{K}_n$ is connected to all other vertices, each vertex is a full-vertex. Therefore, by the result established earlier, we have $\pdim\K[\Delta{\mathcal{K}_n}]=n-1$.
		
		\item The \emph{wheel graph} $W_{n}$ is a graph obtained by connecting a vertex $x$ to all vertices of an induced cycle with $n-1$ vertices, $W_5$ is shown in Figure \ref{fig4}. Since the vertex $x$ in $W_n$ is a full-vertex, we have $\pdim\K[\Delta_{W_n}]=n-1$.
		
		\item A \emph{star complete graph} is obtained by attaching complete graphs  to a single vertex $x$, in each case we attach all the vertices  of the complete graphs to $x$. For example, a bowtie graph can be seen as a star complete graph since it is obtained by attaching two copies of $\mathcal{K}_2$ to a vertex $x$. Then the projective dimension of a star complete graph with $n$ vertices is $n-1$.
		\begin{figure}[ht]
			\begin{center}
				\begin{tikzpicture}[x = 2cm, y = 1.5cm]
					\draw[fill] (0,0) circle [radius = 0.04];
					\draw[fill] (1,0) circle [radius = 0.04];
					\draw[fill] (1.3,.8) circle [radius = 0.04];
					\draw[fill] (.5,1.5) circle [radius = 0.04];
					\draw[fill] (-.3,.8) circle [radius = 0.04];
					\draw[fill] (.5,.7) circle [radius = 0.04];
					\draw[fill] (1.8,0) circle [radius = 0.04];
					
					\draw (0,0)--(.5,.7)--(-.3,.8)--(0,0);
					\draw (1,0)--(.5,.7)--(1.3,.8)--(1.8,0)--(1,0);
					\draw (1.3,.8)--(1,0);
					\draw (.5,1.5)--(.5,.7);
					\draw (1.8,0)--(.5,.7);
				\end{tikzpicture}
				\caption{A star complete graph}\label{fig10}
			\end{center}
		\end{figure}
	\end{itemize}
\end{example}

\begin{remark}
	Let $n\geq6$. Note that $W_n$ is not a co-chordal graph. Therefore, $\K[\Delta_{W_n}]$ does not have a 2-linear resolution. However, we still have $\pdim\K[\Delta_{W_n}]=\max_{1\leq i \leq n}{\deg_{W_n}(x_i)}$.
\end{remark}

The second class of graphs for which the equality holds is the class of graphs whose their complement are $(d_1,d_2,\dots,d_q)$-trees. To our knowledge, this condition along with Theorem \ref{pdim-G1-empty} generalise all the existing classes of graphs for which the equality holds.

\begin{theorem}\label{d-tree-pdim}
Let $G$ be a connected graph such that $\overline{G}$ is a $(d_1,d_2,\dots,d_q)$-tree. Then $\pdim\K[\Delta_G]=\displaystyle\max_{1\leq i \leq n}\{\deg_G(x_i)\}$.
\end{theorem}
\begin{proof}
Let $F_i=V(\mathcal{K}_{d_i})$ for $i=1,\dots,q$. Then $\Delta_G=\lr{F_1,F_2,\dots,F_q}$ is the clique complex of $\overline{G}$. By the definition of a $(d_1, d_2, \dots, d_q)$-tree, there exists a free vertex $x$ in $F_q$ and $\bight I(G)=n-|F_q|$. Note that $ \max_{1\leq i \leq n}\{\deg_G(x_i)\}=\deg_{G}(x)$ because $x$ is a free vertex of $F_q$, and $F_q$ has the minimum cardinality. Let $F_q=\{x,x_2,\dots,x_r\}$. Using Lemmas \ref{max ind} and \ref{deg neighbourhood}, we obtain $\bight I(G)=n-|F_q|=|N_G[F_q]|-|F_q|=|N_G(F_q)|=\deg_{G}(x)+\sum_{i=2}^{r}\deg_{G_{i}}(x_i)$. We have $\dis_G(x)$ is a set of isolated vertices since $x$ is a free vertex in $\Delta_G$. Thus, $\sum_{i=2}^{r}\deg_{G_{i-1}}(x_i)=0$, and $\bight I(G)=\deg_{G}(x)$. Theorem \ref{d-tree-VD} and \cite[Corollary 3.33]{MoreyVillarreal2012} imply that $\pdim\K[\Delta_G]=\bight I(G)=\deg_{G}(x)$.
\end{proof}

\begin{remark}
	It can be observed that if $G$ is the complement of a $(d_1,d_2,\dots,d_q)$-tree, then applying the Auslander--Buchsbaum formula and Theorem \ref{d-tree-pdim} yields $\dep\K[\Delta_G]=d_q$.
\end{remark}

\begin{example}\label{example 2}
	Consider the graphs $G_1$ and $G_2$ in Figure \ref{fig3}. The complement of $G_1$ is a $(2,2,1)$-tree, and the complement of $G_2$ is a $(4,3,3,3,3)$-tree. Then $\pdim\K[\Delta_{G_1}]=3$ $\dep\K[\Delta_{G_1}]=1$. Similarly, $\pdim\K[\Delta_{G_2}]=5$ and $\dep\K[\Delta_{G_2}]=3$.
	\begin{figure}[ht]
		\begin{subfigure}{.45\textwidth}
			\begin{center}
				\begin{tikzpicture}[x = 2cm, y = 1.5cm]
					\draw[fill] (0,0) circle [radius = 0.04];
					\draw[fill] (.5,1) circle [radius = 0.04];
					\draw[fill] (1,0) circle [radius = 0.04];
					\draw[fill] (2,0) circle [radius = 0.04];
					
					\draw (0,0)--(.5,1)--(1,0)--(0,0);
					\draw (1,0)--(2,0);
				\end{tikzpicture}
				\caption*{The graph $G_1$}
			\end{center}
		\end{subfigure}
		\begin{subfigure}{.45\textwidth}
			\begin{center}
				\begin{tikzpicture}[x = 2cm, y = 1.5cm]
					\draw[fill] (0,0) circle [radius = 0.04];
					\draw[fill] (.7,0) circle [radius = 0.04];
					\draw[fill] (.7,.7) circle [radius = 0.04];
					\draw[fill] (0,.7) circle [radius = 0.04];
					\draw[fill] (1.25,.35) circle [radius = 0.04];
					\draw[fill] (-.55,.35) circle [radius = 0.04];
					\draw[fill] (.35,-.7) circle [radius = 0.04];
					\draw[fill] (.35,1.4) circle [radius = 0.04];
					
					\draw (0,0)--(.7,0)--(.7,.7)--(0,.7)--(0,0);
					\draw (.7,0)--(1.25,.35)--(.7,.7)--(.7,0);
					\draw (0,0)--(-.55,.35)--(0,.7)--(0,0);
					\draw (0,.7)--(.35,1.4)--(.7,.7)--(.7,0);
					\draw (0,0)--(.35,-.7)--(.7,0)--(0,0);
					\draw (0,0)--(.7,.7);
					\draw (0,.7)--(.7,0);
				\end{tikzpicture}
				\caption*{The graph $G_2$}
			\end{center}
		\end{subfigure}
		\caption{Graphs $G_1$ and $G_2$ demonstrating $\pdim$ values}\label{fig3}
	\end{figure}
\end{example}

To demonstrate that the converse of Theorem \ref{d-tree-pdim} is not true in general, we can consider the graph $G$ described in Figure \ref{fig4}. We use Theorem \ref{pdim-G1-empty} to determine that $\pdim\K[\Delta_G]= \max_{1\leq i \leq 5}{\deg_G(x_i)}=4$, since $x$ is a full-vertex of $G$. However, we can conclude that the converse of Theorem \ref{d-tree-pdim} does not hold for $G$ since $\overline{G}$ is not a $(d_1,d_2,\dots,d_q)$-tree.
\begin{figure}[ht]
	\begin{subfigure}{.45\textwidth}
		\begin{center}
			\begin{tikzpicture}[x = 2cm, y = 1.5cm]
				\draw[fill] (0,0) circle [radius = 0.04];
				\draw[fill] (0,1) circle [radius = 0.04];
				\draw[fill] (1,1) circle [radius = 0.04];
				\draw[fill] (1,0) circle [radius = 0.04];
				\draw[fill] (.5,.5) circle [radius = 0.04];
				
				\draw (0,0)--(0,1)--(1,1)--(1,0)--(0,0);
				\draw (.5,.5)--(1,1)--(.5,.5)--(1,0);
				\draw (0,1)--(.5,.5);
				\draw (.5,.5)--(0,0);
				
				%					\node[below] at (0,0) {$x_1$};
				%					\node[above] at (0,1) {$x_2$};
				%					\node[above] at (1,1) {$x_3$};
				%					\node[below] at (1,0) {$x_4$};
				\node[below=.05] at (.5,.5) {$x$};
			\end{tikzpicture}
			\caption*{The graph $G$}
		\end{center}
	\end{subfigure}
	\begin{subfigure}{.4\textwidth}
		\begin{center}
			\begin{tikzpicture}[x = 2cm, y = 1.5cm]
				\draw[fill] (0,0) circle [radius = 0.04];
				\draw[fill] (0,1) circle [radius = 0.04];
				\draw[fill] (1,0) circle [radius = 0.04];
				\draw[fill] (1,1) circle [radius = 0.04];
				\draw[fill] (.5,.5) circle [radius = 0.04];
				
				\draw (0,0)--(0,1);
				\draw (1,0)--(1,1);
				%					
				%					\node[below] at (0,0) {$x_1$};
				%					\node[above] at (0,1) {$x_3$};
				%					\node[below] at (1,0) {$x_2$};
				%					\node[above] at (1,1) {$x_4$};
				\node[below] at (.5,.5) {$x$};
			\end{tikzpicture}
			\caption*{The graph of $\overline{G}$}
		\end{center}
	\end{subfigure}
	\caption{A counterexample graph for the converse of Theorem \ref{d-tree-pdim}}\label{fig4}
\end{figure}

The following corollary demonstrates that the theorem of Moradi and Kiani, which states that the projective dimension of $\K[\Delta_G]$ is equal to the maximum degree of vertices of $G$ whenever $\overline{G}$ is a $d$-tree \cite[Theorem 2.13]{MoradiKiani2010}, is a direct consequence of Theorem \ref{d-tree-pdim}.

\begin{corollary}[{\cite[Theorem 2.13]{MoradiKiani2010}}]
	Let $G$ be a graph such that $\overline{G}$ is a $d$-tree. Then $\pdim\K[\Delta_G]=\max_{1\leq i \leq n}\left\{\deg_{G}(x_i)\right\}$.
\end{corollary}
\begin{proof}
	Corollary \ref{d-tree pure VD} implies that any $d$-tree graph is an unmixed $(d_1,d_2,\dots,d_q)$-tree. Since $\overline{G}$ is a $d$-tree, $G$ is connected. Therefore, the result follows from Theorem \ref{d-tree-pdim}.
\end{proof}

Gitler and Valencia defined in \cite{GitlerValencia2005} that for integers $m\geq 1$ and $r\geq 0$, the graph $G_{m,r}$ can be obtained by attaching $r$ edges to each vertex of the complete graph $\mathcal{K}_m$. They showed that the projective dimension of the Stanley--Reisner ring $\K[\Delta{G_{m,r}}]$ is equal to the maximum degree of the vertices in $G_{m,r}$, where $n=m(r+1)$ is the number of vertices in $G_{m,r}$. Theorem \ref{d-tree-pdim} generalizes this result, as the complement of $G_{m,r}$ is a $(mr, \underbrace{(m-1)r+1,(m-1)r+1,\dots,(m-1)r+1}_{m\textup{-times}})$-tree. 

\begin{corollary}
	[{\cite[Proposition 4.9 and Proposition 4.12]{GitlerValencia2005}}]
	Let $G=G_{m,r}$ or let $G=G_{m,i_1,\dots,i_m}$ be the graph obtained by attaching $i_j$ edges at each vertex $x_j$ of $\mathcal{K}_m$ such that $i_j\leq i_{j+1}$ for all $j$. Then $\pdim\K[\Delta_G]=\displaystyle\max_{1\leq i \leq n}\{\deg_{G}(x_i)\}$.
\end{corollary}

	A \emph{separating set} in a connected graph $G$ is a set of edges whose deletion turns $G$ into a disconnected graph. If $G$ is connected and not a complete graph, its \emph{edge connectivity} $\lambda(G)$ is the size of the smallest separating set in $G$. A graph $G$ is $k$-connected if $\lambda(G)\geq k$.
	
	\begin{lemma}\label{k-connected}
		For $q\geq2$, let $G$ be a connected $(d_1,d_2,\dots,d_q)$-tree. Then $G$ is $d_q-1$-connected.
	\end{lemma}
	\begin{proof}
		Let $x$ be the free vertex in  the facet $F_q=V(\mathcal{K}_{d_q})$ of $\Delta(G)$. Then $\lambda(G)=d_q-1$ since $F_q$ has the minimum cardinality and $G$ is a chordal graph. Therefore $G$ is $d_q-1$-connected. 
	\end{proof}
	
	\begin{example}
		Consider the graphs $G_1$ and $G_2$ in Figure \ref{fig11}. Then $G_1$ is a $(3,2)$-tree and $G_2$ is a $(3,3)$-tree. Therefore, according to Lemma \ref{k-connected}, $G_1$ is 1-connected, while $G_2$ is 2-connected.
		\begin{figure}[ht]
			\begin{subfigure}{.45\textwidth}
				\begin{center}
					\begin{tikzpicture}[x = 2cm, y = 1.5cm]
						\draw[fill] (0,0) circle [radius = 0.04];
						\draw[fill] (.5,1) circle [radius = 0.04];
						\draw[fill] (1,0) circle [radius = 0.04];
						\draw[fill] (2,0) circle [radius = 0.04];
						
						\draw (0,0)--(.5,1)--(1,0)--(0,0);
						\draw (1,0)--(2,0);
					\end{tikzpicture}
					\caption*{The graph $G_1$}
				\end{center}
			\end{subfigure}
			\begin{subfigure}{.45\textwidth}
				\begin{center}
					\begin{tikzpicture}[x = 2cm, y = 1.5cm]
						\draw[fill] (0,0) circle [radius = 0.04];
						\draw[fill] (1,1) circle [radius = 0.04];
						\draw[fill] (1,0) circle [radius = 0.04];
						\draw[fill] (2,1) circle [radius = 0.04];
						
						\draw (0,0)--(1,1)--(1,0)--(0,0);
						\draw (1,0)--(1,1)--(2,1)--(1,0);
					\end{tikzpicture}
					\caption*{The graph $G_2$}
				\end{center}
			\end{subfigure}
			\caption{Graphs $G_1$ and $G_2$ illustrating their connectivity}\label{fig11}
		\end{figure}
	\end{example}
	
	To compute the projective dimension of the quasi-forest simplicial complex $\Delta_G$ for a disconnected graph $G$, we can split $G$ into a connected component and a union of totally disconnected graphs. Hence, after dropping the connectedness assumption in Theorem \ref{d-tree-pdim}, our formula for computing the projective dimension becomes
	\begin{equation*}
		\pdim\K[\Delta_G]=\displaystyle\max_{1\leq i \leq n}\{\deg_{G}(x_i)\}+\text{number of isolated vertices}.
	\end{equation*}
	
	For example, consider the graph $G$ shown in Figure \ref{fig9}. Note that $G$ is disconnected and has two connected components: $G'$ and the isolated vertex $x$. By Theorem \ref{d-tree-pdim}, we have $\pdim\K[\Delta_{G'}]=3$ since $\overline{G'}$ is a $(2,1,1)$-tree. Thus, by the formula given above, we have $\pdim\K[\Delta_G]=3+1=4$.\\
	\begin{figure}[ht]
		\begin{center}
			\begin{tikzpicture}[x = 2cm, y = 1.5cm]
				\draw[fill] (0,0) circle [radius = 0.04];
				\draw[fill] (1,1) circle [radius = 0.04];
				\draw[fill] (1,0) circle [radius = 0.04];
				\draw[fill] (0,1) circle [radius = 0.04];
				\draw[fill] (2,0) circle [radius = 0.04];
				
				\draw (0,0)--(1,0)--(1,1)--(0,1)--(0,0);
				\draw (1,0)--(0,1);
				
				\node[above] at (2,0) {$x$};
			\end{tikzpicture}
		\end{center}
		\caption{The graph $G$ with an isolated vertex}\label{fig9}
	\end{figure}

\textbf{Data availability.} Authors can confirm that all relevant data are included in the article.

\bibliographystyle{plain}
\bibliography{SCM_trees.bib}

\end{document}